\documentclass{article}
\usepackage{amssymb,amsthm,amsmath,enumerate}
\usepackage{graphics}

\theoremstyle{plain}

\newtheorem{theorem}{Theorem}

\newcommand*{\R}{\ensuremath{\mathbb{R}}}

\newcommand*{\C}{\ensuremath{\mathbb{C}}}
\newcommand*{\N}{\ensuremath{\mathbb{N}}}

\begin{document}

\date{}

\author{{\.Zywilla Fechner and L\'aszl\'o Sz\'ekelyhidi}\\ 
{\small\it Institute of Mathematical Finance, Ulm University}, \\{\small\it Helmholtzstrasse 18, 89081 Ulm, Germany,}\\
   {\small\rm e-mail: \tt zfechner@gmail.com}\\
   {\small\it Institute of Mathematics, University of Debrecen,}\\
   {\small\rm e-mail: \tt lszekelyhidi@gmail.com} }

\title{Sine functions on hypergroups
   \footnotetext{The research was partly supported by the
   Hungarian National Foundation for Scientific Research (OTKA),
   Grant No. K111651.}\footnotetext{AMS (2000)
   Subject Classification: 20N20, 43A62, 39B99}}

\maketitle

\begin{abstract} 
In a recent paper we introduced sine functions on commutative hypergroups. These functions are natural generalizations of those functions on groups which are products of additive and multiplicative homomorphisms. In this paper we describe sine functions on different types of hypergroups, including polynomial hypergroups, Sturm--Liouville hypergroups, etc. A non-commutative hypergroup is also considered.  
\end{abstract}

\vskip1cm

\section{Introduction}

\hskip.5cm In \cite{FeSz15} we introduced the concept of sine functions on commutative hypergroups and utilized this concept for the description of the solutions of some sine and cosine functional equations. In this paper we shall describe sine functions on some general classes of hypergroups. In th the sequel $\C$ denotes  the set of complex numbers. By a {\it hypergroup} we mean a locally compact hypergroup. The identity element of the hypergroup $K$ will be denoted by $o$.
\vskip.3cm

For basics about hypergroups see the monograph \cite{BlH95}. The detailed study of functional equations on hypergroups started with the papers \cite{MR2042564, MR2107959, MR2161803}. A comprehensive monograph on the subject is \cite{Sze12}. 
\vskip.3cm

Let $K$ be a hypergroup and $m$ an exponential on $K$, that is, a non-identically zero continuous function $m:K\to \C$ satisfying $m(x*y)=m(x) m(y)$ for each $x,y$ in $K$. Exponentials play an important role on hypergroups, in particular in the commutative case. The description of exponentials on different types of commutative hypergroups can be found in \cite{Sze12}. Here we shall consider non-commutative cases, too.
\vskip.3cm

The continuous function $f:K\to\C$ will be called an {\it $m$-sine function}, if it satisfies
\begin{equation}\label{msine}
f(x*y)= f(x) m(y)+f(y) m(x)
\end{equation}
for each $x,y$ in $K$. The function $f$ is called {\it sine function} if it is an $m$-sine function for some exponential $m$. Clearly, every sine function $f$ satisfies \hbox{$f(o)=0$.} Obviously, $m\equiv 1$ is an exponential on any hypergroup, and $1$-sine functions are called {\it additive functions}. If the hypergroup operation on $K$ arises from a group operation, say $K=G$ with some locally compact group, then $m$-sine functions can be described completely, as it is shown by the following theorem

\begin{theorem}\label{group}
Let $G$ be a locally compact group and $m:G\to\C$ an exponential on $G$. Then every $m$-sine function $f:G\to\C$ has the form $f= a\cdot m$, where $a:G\to\C$ is an additive function.
\end{theorem}

\begin{proof}
If $m$ is an exponential on $G$, then $m(x)\ne 0$ for each $x$ in $G$. Indeed, if $m(x)=0$ for some $x$ in $G$, then for each $y$ in $G$ we have
$$
m(y)=m(y\cdot x_0^{-1} \cdot x_0)=m(y\cdot x_0^{-1})\cdot m( x_0)=0,
$$
a contradiction. Now for given $x,y$ in $G$ we divide equation \eqref{msine} by $m(x y)$, then we have
$$
\frac{f(x y)}{m(x y)}=\frac{f(x)}{m(x)}+\frac{f(y)}{m(y)},
$$
which means that $\frac{f}{m}$ is additive, and the statement is proved.
\end{proof}

The interest of the non-group case lies in the fact that exponentials on hypergroups may take the value $0$. In this case, in general, equation \eqref{msine} cannot be reduced to the functional equation of additive functions, and -- depending on the special structure of the given hypergroup -- new types of sine functions occur. However, it is clear that for each exponential $m$ all $m$-sine functions form a complex linear space.
\vskip.3cm

In the subsequent sections we shall study sine functions in general, and we shall describe them on some special types of hypergroups.

\section{Sine functions on compact hypergroups}
\hskip.5cm Additive functions on compact hypergroups are obviously identically zero, as, by continuity, the range of an additive function on a compact hypergroup is an additive subsemigroup of the additive group of complex numbers, \hbox{hence it is $\{0\}$.} This means that $1$-sine functions on compact hypergroups are trivial -- it is reasonable to ask if the same holds for every $m$-sine function. The following theorem gives a partial answer to this question.  

\begin{theorem}
Let $K$ be a compact hypergroup and $m:K\to\C$ an exponential. Then for each $m$-sine function on $K$ we have $f\cdot m=0$.
\end{theorem}

\begin{proof}
By induction we have
$$
f(x*y^n)=f(x) m(y)^n+n f(y) m(x) m(y)^{n-1}
$$
for each $x,y$ in $K$ and positive integer $n$. Here $y^n$ is meant in the sense of convolution: $y^1=y$, and $y^{n+1}=y^n*y$. Putting $x=o$ in the above equation we have
$$
f(y^n)=n f(y) m(y)^{n-1}
$$
for each $y$ in $K$ and positive integer $n$. For a given $y$ in $K$ the left hand side is bounded and the right hand side is unbounded as a function of $n$ unless $f(y) m(y)=0$.
\end{proof}

In particular, finite hypergroups are compact. The following theorem gives some insight in this particular case.

\begin{theorem}\label{Dtheta}
Let $K=D(\theta)$ be the two-point hypergroup with $0<\theta<1$. Then every sine function on $K$ is identically zero.
\end{theorem}

\begin{proof}
It is known (see e.g. \cite{Sze12}) that there are two exponentials on $K$, namely $m_0, m_1:K\to\C$, where $m_0\equiv 1$, and $m_1(0)=1$, $m_1(1)=-\theta$. Hence $m_0$-sine functions are the additive functions, which are identically zero, by the above considerations. Let $f:K\to\C$ be an $m_1$-sine function, then $f(o)=0$ and we have
$$
(1-\theta) f(1)=\theta f(0)+(1-\theta) f(1)=f(1*1)=2 f(1) m_1(1)= -2 \theta f(1),
$$
which implies $(1+\theta) f(1)=0$, hence $f(1)=0$.
\end{proof}

Nevertheless, the general problem is still open: is it true that every sine function on a compact hypergroup is identically zero? 

\section{Sine functions on polynomial hypergroups in one variable}

\hskip.5cm Before we study sine functions on polynomial hypergroups we present a general theorem. We recall (see \cite{Sze12}) that given a hypergroup $K$ and a positive integer $n$ the function $\Phi:K\times \C^n$ is called an {\it exponential family}, if 
\begin{enumerate}[(1)]
\item The function $x\mapsto \Phi(x,\lambda)$ is an exponential on $K$ for each $\lambda$ in $\C^n$.
\item The function $\lambda\mapsto \Phi(x,\lambda)$ is entire for each $x$ in $K$.
\item For each exponential $m$ on $K$ there exists a $\lambda$ in $\C^n$ such that $m(x)=\Phi(x,\lambda)$ for each $x$ in $K$. 
\end{enumerate}
In the presence of exponential families the following theorem exhibits a class of sine functions.

\begin{theorem}\label{expfam}
Let $K$ be a hypergroup with the exponential family $\Phi$. Then the function $x\mapsto \partial_{k+1} \Phi(x,\lambda)$ is an $m_{\lambda}$-sine function for each $\lambda=(\lambda_1,\lambda_2,\dots,\lambda_n)$ in $\C^n$ and for $k=1,2,\dots,n$, where $m_{\lambda}(x)=\Phi(x,\lambda)$.
\end{theorem}

\begin{proof}
By definition, we have
$$
\Phi(x*y,\lambda_1,\lambda_2,\dots,\lambda_n)=\Phi(x,\lambda_1,\lambda_2,\dots,\lambda_n)\cdot \Phi(y,\lambda_1,\lambda_2,\dots,\lambda_n)
$$
holds for each $x,y$ in $K$ and $\lambda=(\lambda_1,\lambda_2,\dots,\lambda_n)$ in $\C^n$. Then we have for each $k=1,2,\dots,n$ 
$$
\partial_{k+1} \Phi(x*y,\lambda)=\partial_{k+1} \Phi(x,\lambda) \, \Phi(y,\lambda)+\Phi(x,\lambda) \partial_{k+1} \,\Phi(y,\lambda)
$$
which was to be proved.
\end{proof}

The following theorem describes all sine functions on polynomial hypergroups in a single variable.

\begin{theorem}\label{polyone}
Let $K$ be the polynomial hypergroup generated by the sequence of polynomials $(P_n)_{n\in\N}$. Then every sine function on $K$ has the form $n\mapsto c\,P_n'(\lambda)$ with some complex numbers $c,\lambda$.
\end{theorem}

\begin{proof}
The sufficiency of the condition is obvious by the previous theorem: the function $n\mapsto c\,P_n'(\lambda)$ is clearly an $m_{\lambda}$-sine function for each complex $c,\lambda$, if $m_{\lambda}(n)=P_n(\lambda)$. Now we prove the converse statement.
\vskip.3cm

Let $\lambda$ be a complex number and let $m_{\lambda}$ denote the exponential $n\mapsto P_n(\lambda)$ on $K$, further let $f:K\to\C$ be an $m_{\lambda}$-sine function. Hence we have $f(0)=0$ and
\begin{equation}\label{0poly}
f(n*k)-f(n) P_k(\lambda)-f(k) P_n(\lambda)=0
\end{equation}
for each $k,n$ in $\N$. Substitution $k=1$ gives
\begin{equation}\label{1poly}
f(n*1)-\lambda f(n)- f(1) P_n(\lambda)=0
\end{equation}
for each $n=0,1,\dots$. Now we define the function $g(n)=f(1)\cdot P_n'(\lambda)$, then $g:K\to\C$ is an $m_{\lambda}$-sine function, by Theorem \ref{expfam}, hence it satisfies \eqref{0poly}, too:
\begin{equation*}
g(n*k)-g(n) P_k(\lambda)-g(k) P_n(\lambda)=0
\end{equation*}
for each $k,n$ in $\N$. In particular, we also have
\begin{equation}\label{2poly}
g(n*1)-\lambda g(n)- g(1) P_n(\lambda)=0
\end{equation}
for each $n=0,1,\dots$. On the other hand, $g(1)=f(1) P_1'(\lambda)=f(1)$, and, by \eqref{1poly} and \eqref{2poly}, the function $\varphi=f-g$ satisfies
\begin{equation}\label{3poly}
\varphi(n*1)-\lambda \varphi(n)=0
\end{equation}
for each $n=0,1,\dots$ with $\varphi(0)=\varphi(1)=0$. As \eqref{3poly} is a linear homogeneous difference equation of order $2$ it follows immediately that $\varphi=0$, and $f=g$ and the theorem is proved.
\end{proof}

This theorem presents a variety of sine $m$-functions which are not of the form $a\cdot m$ with some additive function $a$. Indeed, on the polynomial hypergroup $K$ every additive function has the form $n\mapsto c\cdot P_{n}'(1)$ and every exponential has the form $n\mapsto P_n(\lambda)$ (see e.g. \cite[Theorem 2.2, Theorem 2.3]{Sze12}) with some complex numbers $c,\lambda$. However, it is not true in general that $P_{n}'(\lambda)=c P_{n}'(1) P_n(\lambda)$ holds for each $n$ in $\N$ with some complex number $c$.

\section{Sine functions on the $SU(2)$ hypergroup}

\hskip.5cm In this section we describe the sine functions on the $SU(2)$ hypergroup (see e.g. \cite{BlH95, Sze12}). The method we use here is similar to that in the previous section.
\vskip.3cm

We recall that the $SU(2)$ hypergroup structure is introduced on the set $\N$ of natural numbers, where convolution of the point masses $\delta_k$ and $\delta_n$ is defined in the following manner:
$$
\delta_k*\delta_n={\sum_{l=|k-n|}^{k+n}} ' \,\frac{l+1}{(k+1) (n+1)} \delta_l
$$ 
where the prime means that only every second term appears in the sum. Here $\delta_k$ denotes the point mass concentrated at $k$ ($k=0,1,\dots$). This hypergroup has the exponential family $\Phi:K\times \C\to\C$ given by
\begin{equation}\label{sexpfam}
\Phi(n,\lambda)=\frac{\sinh [(n+1) \lambda]}{(n+1) \sinh \lambda}
\end{equation}
for each $n$ in $\N$ and $\lambda$ in $\C$ with $\Phi(n,0)=1$ (see \cite{Sze12}).
\vskip.3cm

We have the following result. 

\begin{theorem}\label{Su2}
Let $K$ be the $SU(2)$ hypergroup and for each complex number $\lambda$ let $m_{\lambda}$ denote the exponential $n\mapsto \Phi(n,\lambda)$ with $\Phi$ given in \eqref{sexpfam}. Then every $m_{\lambda}$-sine function on $K$ has the form $n\mapsto c\, \partial_2 \Phi(n,\lambda)$ with some complex \hbox{number $c$.}
\end{theorem}

\begin{proof}
By Theorem \eqref{expfam}, every function of the given form is an $m_{\lambda}$-sine function. For the converse, by the definition of the convolution in $K$ we have that the $m_{\lambda}$-sine function $f:K\to\C$ satisfies
\begin{equation*}
f(n*1)={\sum_{l=n-1}^{n+1}}'\, \frac{l+1}{2 (n+1)}\,f(l)=\frac{n}{2 (n+1)}\,f(n-1)+\frac{n+2}{2 (n+1)}\,f(n+1)=
\end{equation*}
$$
f(n) m_{\lambda}(1)+f(1) m_{\lambda}(n)
$$
for each $n=1,2,\dots$. We can write this equation in the following form
\begin{equation}\label{Su21}
(n+3) f(n+2)- 2(n+2) \cosh \lambda f(n+1)+(n+1) f(n)= 2 f(1) (n+2) m_{\lambda}(n+1)
\end{equation}
for $n=0,1,\dots$. We introduce $g(n)=(n+1) f(n)$, then we have
\begin{equation}\label{Su22}
g(n+2)- 2\cosh \lambda g(n+1)+ g(n)= 2 f(1) (n+2) m_{\lambda}(n+1)\,.
\end{equation}
By Theorem \eqref{expfam}, the function $n\mapsto \partial_2 \Phi(n,\lambda)$ is an $m_{\lambda}$-sine function, hence it also satisfies \eqref{Su21} and the function $\varphi(n)= (n+1) \partial_2 \Phi(n,\lambda)$ satisfies \eqref{Su22}:
\begin{equation}\label{Su23}
\varphi (n+2)- 2\cosh \lambda \varphi (n+1)+ \varphi (n)= 2 a (n+2) m_{\lambda}(n+1)\,,
\end{equation}
where $a=\partial_2 \Phi(1,\lambda)$. Multiplying \eqref{Su22} by $a$ and \eqref{Su23} by $f(1)$, then subtracting the two equations we have that the function $\psi=a g-f(1) \varphi$ satisfies
\begin{equation}\label{Su24}
\psi (n+2)- 2\cosh \lambda \psi (n+1)+ \psi (n)= 0
\end{equation}
for $n=0,1,\dots$. On the other hand, we have $\psi(0)=a g(0)-f(1) \varphi(0)=0$, and 
$$
\psi(1)=a g(1)-f(1) \varphi(1)=2 a f(1)-2 f(1) a=0\,,
$$
hence $\psi=0$. As it is easy to check that $a\ne 0$ our statement follows. 
\end{proof}

\section{Sine functions on polynomial hypergroups in several variables}

\hskip.5cm In this section we consider the case of polynomial hypergroups in several variables. It is known (see e.g. \cite{Sze12}) that if $K$ is a $d$-dimensional polynomial hypergroup generated by the family of polynomials $(Q_x)_{x\in K}$, then $K$ has the exponential family $\Phi:K\times \C^d\to\C$ with
$$
\Phi(x,\lambda)=Q_x(\lambda)
$$
for each $x$ in $K$ and $\lambda$ in $\C^n$ (see \cite[Theorem 3.1]{Sze12}). Based on this result we have the following theorem.

\begin{theorem}\label{sinsev}
Let $K$ be a $d$-dimensional polynomial hypergroup generated by the family of polynomials $(Q_x)_{x\in K}$ and for each $\lambda$ in $\C^d$ let $m_{\lambda}$ denote the exponential $x\mapsto Q_x(\lambda)$ on $K$. The function $f:K\to\C$ is an $m_{\lambda}$-sine function on $K$ if and only if there are complex numbers $c_j$, $j=1,2,\dots,d$ such that
$$
f(x)=\sum_{j=1}^d c_j \,\partial_j Q_x(\lambda)
$$
holds for each $x$ in $K$.
\end{theorem} 

\begin{proof}
By Theorem \ref{expfam}, every function $f$ of the given form is an $m_{\lambda}$-sine function on $K$.\vskip.3cm

To prove the converse first we note that for each $n$ in $\N$ we denote by $K_n$ the set of all polynomials of degree at most $n$ in $K$. By the definition of the hypergroup, the family of polynomials $(Q_x)_{x\in K_n}$ is a basis of $K_n$. In particular, the family of polynomials $Q_x$ with $x$ in $K_1$ form a basis of the linear polynomials in $K$. It follows that the vectors
$$
\bigl(\partial_1 Q_x(\lambda),\partial_2 Q_x(\lambda),\dots,\partial_d Q_x(\lambda)\bigr)
$$
for $x\ne $ in $K_1$ are linearly independent. This implies that the system of linear equations
$$
f(x)=\sum_{j=1}^d c_j \partial_j Q_x(\lambda)
$$
with $x\ne o$ in $K_1$ has a unique solution $c_1,c_2,\dots,c_d$. Then, clearly, this holds also for $x=o$. We prove by induction on $n$ that this equation holds for each $x$ in $K_n$. Assuming that it holds for $n$ let $x$ be in $K_{n+1}$. Then there are elements $x_l$ in $K_1$ and $y_l$ in $K_n$ for $l=1,2,\dots,s$ such that
\begin{equation}\label{part}
Q_x(\xi)=\sum_{l=1}^s a_l\, Q_{x_l}(\xi)\,Q_{y_l}(\xi)
\end{equation}
holds for each $\xi$ in $\C^d$ with some complex numbers $a_l$, $l=1,2,\dots,s$. By the definition of the convolution in $K$, this implies
$$
\delta_x=\sum_{l=1}^s a_l\, \delta_{x_l}*\delta_{y_l}\,.
$$
On the other hand, applying $\partial_j$ on equation \eqref{part} and then substituting $\xi=\lambda$ we obtain
$$
\partial_j Q_x(\xi)=\sum_{l=1}^s a_l\,[\partial_j Q_{x_l}(\xi)\,Q_{y_l}(\xi)+Q_{x_l}(\xi)\,\partial_j Q_{y_l}(\xi)]\,.
$$
Combining these results we conclude
$$
f(x)=\int_K f\,d \delta_x=\sum_{l=1}^s a_l\,\int_K f d(\delta_{x_l}*\delta_{y_l})=
$$
$$
\sum_{l=1}^s a_l\,f (x_l*y_l)=\sum_{l=1}^s a_l [f(x_l) m_(y_l)+f(y_l) m(x_l)]=
$$
$$
\sum_{l=1}^s a_l \sum_{j=1}^d c_j [\partial_j Q_{x_l}(\lambda) Q_{y_l}(\lambda)+ Q_{x_l}(\lambda) \partial_j Q_{y_l}(\lambda)]=
$$
$$
\sum_{j=1}^d c_j \sum_{l=1}^s a_l [\partial_j Q_{x_l}(\lambda) Q_{y_l}(\lambda)+ Q_{x_l}(\lambda) \partial_j Q_{y_l}(\lambda)]=\sum_{j=1}^d c_j\,\partial_j Q_x(\lambda)\,,
$$
which proves our theorem.
\end{proof}

\section{Sine functions on Sturm--Liouville \\hypergroups}

\hskip.5cm In this section we shall study sine functions on Sturm--Liouville hypergroups. For basics on Sturm--Liouville hypergroups the reader should consult with \hbox{\cite{BlH95, Sze12}.} We recall that if $K$ is a Sturm--Liouville hypergroup with the Sturm--Liouville function $A$, then there exists an exponential family $\Phi:K\times C\to \C$ satisfying
\begin{equation}\label{Sturm1}
\partial_1^2\Phi(x,\lambda)+\frac{A'(x)}{A(x)} \,\partial_1\Phi(x,\lambda)=\lambda\,\Phi(x,\lambda)
\end{equation}
for each $x>0$ and $\lambda$ in $\C$ with the initial conditions $\Phi(0,\lambda)=1$, $\partial_1\Phi(0,\lambda)=0$. Our next results characterizes sine functions on Sturm--Liouville hypergroups in terms of an initial value problem.

\begin{theorem}\label{Sturmsine1}
Let $K$ be a Sturm--Liouville hypergroup with the Sturm--Liouville function $A$ and let $\Phi:K\times \C\to\C$ be its exponential family. For each complex number $\lambda$ let $m_{\lambda}$ denote the exponential $x\mapsto \Phi(x,\lambda)$. The  function $f:K\to\C$ is an $m_{\lambda}$-sine function on $K$ if and only if it is twice continuously differentiable for $x>0$ and there exists a complex number $c$ such that $f$ satisfies
\begin{equation}\label{Sturm2}
f''(x)+\frac{A'(x)}{A(x)} \,f'(x)=\lambda\,f(x)+c\,\Phi(x,\lambda)
\end{equation}
for $x>0$, with the initial conditions $f(0)=0$, $f'(0)=0$.
\end{theorem}

\begin{proof}
First we assume that $f:K\to\C$ is an $m_{\lambda}$-sine function on $K$, then it satisfies $f(0)=0$ and 
\begin{equation}\label{Sturm3}
f(x*y)=f(x) \Phi(y,\lambda)+f(y) \Phi(x,\lambda)
\end{equation}
for each $x,y\geq 0$. By the definition of the Sturm--Liouville hypergroup the function $u_f$ defined by $u_f(x,y)=f(x*y)$ is twice continuously differentiable for $x,y>0$ and it satisfies 
\begin{equation}\label{Sturmdef}
\partial_1^2 u_f(x,y)+\frac{A'(x)}{A(x)}\,\partial_1 u_f(x,y)=\partial_2^2 u_f(x,y)+\frac{A'(y)}{A(y)}\,\partial_2 u_f(x,y)
\end{equation}
further $\partial_2 u_f(x,0)=0$. From the general properties of hypergroups it follows also $u_f(x,0)=u_f(0,x)=f(x)$ and $\partial_1 u_f(0,y)=0$. By \eqref{Sturm3} and \eqref{Sturmdef} this gives
$$
\Phi(x,\lambda)\Bigl[f''(y)+\frac{A'(y)}{A(y)}\,f'(y)-\lambda f(y)\Bigr]=
\Phi(y,\lambda)\Bigl[f''(x)+\frac{A'(x)}{A(x)}\,f'(x)-\lambda f(x)\Bigr]
$$
for each $x,y>0$. As $\Phi$ is not identically zero we conclude that \eqref{Sturm2} holds for each $x>0$. The condition $f'(0)=0$ is the consequence of $\partial_2 u_f(x,0)=0$.
\vskip.3cm

Conversely, suppose that $f:K\to\C$ is twice continuously differentiable for $x>0$ and satisfies \eqref{Sturm2} with the given initial conditions. The function $u_f$ defined above obviously satisfies the conditions given in \eqref{Sturmdef}. If we define
$$
u(x,y)=f(x) \Phi(y,\lambda)+f(y) \Phi(x,\lambda)
$$
for $x,y>0$ and $u(x,0)=u(0,x)=f(x)$, then, by \eqref{Sturm2}, it follows that also $u$ satisfies the same conditions given in \eqref{Sturmdef}. By uniqueness, we conclude that $u=u_f$, hence $f$ is an $m_{\lambda}$-sine function. 
\end{proof}

The description of sine functions on Sturm--Liouville hypergroups follows.

\begin{theorem}\label{Sturmsine2}
Let $K$ be a Sturm--Liouville hypergroup with exponential family $\Phi:K\times \C\to\C$ and for each $\lambda$ complex number let $m_{\lambda}$ denote the exponential $x\mapsto \Phi(x,\lambda)$. The function $f:K\to\C$ is an $m_{\lambda}$-sine function on $K$ if and only if there is a complex number $a$ such that 
$$
f(x)=a\,\partial_2 \Phi(x,\lambda)
$$
holds for each $x$ in $K$.
\end{theorem}

\begin{proof}
By Theorem \ref{expfam}, every function of the given form is an $m_{\lambda}$-sine function.
\vskip.3cm

To prove the converse suppose that $f:K\to\C$ is an $m_{\lambda}$-sine function. Then, by the previous theorem, $f$ is twice differentiable and it satisfies the differential equation \eqref{Sturm2} with some complex number $c$ together with the given initial conditions. On the other hand, by Theorem \ref{expfam}, the function $g:K\to\C$ defined by
$$
g(x)=\partial_2 \Phi(x,\lambda)
$$
for each $x$ in $K$ also satisfies the same equation with some other complex number $d$ and with the same initial conditions. In other words, we have
\begin{equation}\label{Sturm2a}
f''(x)+\frac{A'(x)}{A(x)} \,f'(x)=\lambda\,f(x)+c\,\Phi(x,\lambda)
\end{equation}
and
\begin{equation}\label{Sturm2b}
g''(x)+\frac{A'(x)}{A(x)} \,g'(x)=\lambda\,g(x)+d\,\Phi(x,\lambda)
\end{equation}
for each $x>0$. Obviously, $d\ne0$, otherwise $g=0$, which is not the case. Multiplying \eqref{Sturm2a} by $d$ and \eqref{Sturm2b} by $c$, then subtracting the two equations we obtain
\begin{equation}\label{Sturm2c}
(d f-c g)''(x)+\frac{A'(x)}{A(x)} \,(d f-c g)'(x)=\lambda\,(d f-c g)(x)
\end{equation}
for each $x>0$. Further we have that $(d f-c g)(0)=(d f-c g)'(0)=0$. By uniqueness we have $d f=c g$, which gives the statement with $a=\frac{c}{d}$.
\end{proof}

Similarly to Theorem \ref{polyone} this result exhibits sine functions which are not of the form of a product of an additive and an exponential function. 

\section{Sine functions on a non-commutative\\ hypergroup}

\hskip.5cm In this section we describe sine functions on a non-commutative hypergroup. We consider the multiplicative group of matrices of the form 
$$
\Bigl(
\begin{array}{cc}
x & u \\ 
0 & 1
\end{array} 
\Bigr)
$$
where $x,u$ are real numbers, $x\ne 0$. All these matrices form a subgroup of $GL(2)$ which can be identified with a subset of $\R^2$ and it is a locally compact topological group $G$ when equipped with the topology inherited from $\R^2$. As we have
$$
\Bigl(
\begin{array}{cc}
x & u \\ 
0 & 1
\end{array} 
\Bigr) \cdot \Bigl(
\begin{array}{cc}
y & v \\ 
0 & 1
\end{array} 
\Bigr) = \Bigl(
\begin{array}{cc}
x y & x u+v \\ 
0 & 1
\end{array} 
\Bigr)
$$
we can describe the group operation on the set $G=\{(x,u):\, x,u\in \R, x\ne 0\}$ in the following way:
$$
(x,u) \cdot (y,v)= (x y, x v+u)\,.
$$
Obviously, this group is non-commutative. The identity is $(1,0)$, and the inverse of $(x,u)$ is 
$$
(x,u)^{-1}=(\frac{1}{x}, -\frac{u}{x})\,.
$$
For more about this group see e.g. \cite{HeR79}, p.~201.
\vskip.3cm

Let $K$ be a compact subgroup of $G$ and let $\omega$ denote the Haar measure on $K$ -- as $K$ is unimodular, hence $\omega$ is both left and right invariant, further it is inversion invariant. We shall consider the {\it double cosets} of $K$ which are defined as the sets $K (x,u) K=\{ k (x,u) l:\, k,l\in K\}$ for each $(x,u)$ in $G$. 
\vskip.3cm

We introduce a hypergroup structure on the set $L=G/ /K$ of all double cosets in the following way. The topology of $L$ is the factor topology, which is locally compact. The identity is the coset $K$ itself and the involution is defined by
$$
(K (x,u) K)^{\vee}=K (x,u)^{-1} K\,.
$$
Finally, the convolution of $\delta_{K (x,u) K}$ and $\delta_{K (y,v) K}$ is defined by
$$
\delta_{K (x,u) K}*\delta_{K (y,v) K}=\int_K \delta_{K (x,u) k (y,v) K}\,d\omega(k)\,.
$$
It is known that this gives a hypergroup structure on $L$ (see \cite{BlH95}, p.~12.), which is non-commutative, in general. If $K$ is a normal subgroup, then $L$ is isomorphic to the hypergroup arising from the factor group $G/K$. But now we choose a non-normal subgroup $K$ in the following way: let $K$ be the two-element set $\{(1,0), (-1,0)\}$. Here we have $(-1,0)\cdot (-1,0)=(1,0)$, hence $K$ is a compact subgroup of $G$. On the other hand, we have
$$
(x,u)\cdot (-1,0)\cdot (x,u)^{-1}= (-1,2u)
$$
for each $(x,u)$ in $G$, which shows that $K$ is not a normal subgroup. 
\vskip.3cm

It is easy to see that the element $(y,v)$ is in the double coset $K (x,u) K$ if and only if it is either of the following elements: $(x,u)$, $(-x,u)$, $(x,-u)$, $(-x,-u)$. Further, by the definition of convolution on $L$, for every continuous function $f:L\to\C$ we have 
\begin{equation*}
f\Bigl[ K (x,u) K* K (y,v) K\Bigr]=\frac{1}{2} f\Bigl[K (x y, x v+u) K\Bigr]+\frac{1}{2} f\Bigl[K (-x y, -x v+u) K\Bigr]\,.
\end{equation*} 
We note that complex valued functions on $L$ can be identified with those complex valued functions on $G$ satisfying the {\it compatibility conditions}
\begin{equation}\label{comp}
f(x,u)=f(-x,u)=f(x,-u)=f(-x,-u)
\end{equation}
for each $(x,u)$ in $G$. Using this identification the above equation can be written -- in a somewhat loose way -- in the form
\begin{equation}\label{convL}
f\bigl[(x,u) *  (y,v) \bigr]=\frac{1}{2} f (x y, x v+u)+\frac{1}{2} f (-x y, -x v+u)\,.
\end{equation}
It follows that the exponentials on $L$ can be identified with those non-identically zero continuous functions $m:G\to \C$ satisfying the compatibility conditions \eqref{comp}
\begin{equation*}
m(x,u)=m(-x,u)=m(x,-u)=m(-x,-u)
\end{equation*}
further 
\begin{equation*}
m(x y, x v+u)+m (-x y, -x v+u)=2 m(x,u) m(y,v)
\end{equation*}
for each $(x,u), (y,v)$ in $G$. By the compatibility conditions \eqref{comp} the latter equation can be written as 
\begin{equation}\label{expL}
m(x y, x v+u)+m (x y, x v-u)=2 m(x,u) m(y,v)\,.
\end{equation}
Putting $y=1$ and $v=0$ we obtain that $m(1,0)=1$ and with $u=v=0$ we have 
$$
m(x y,0)=m(x,0)\, m(y,0)
$$
for each nonzero real $x,y$. It follows $m(x,0)=|x|^{\lambda}$ with some complex number $\lambda$. On the other hand, putting $x=y=1$ in \eqref{expL} we have
$$
m(1,u+v)+m(1,u-v)=2m(1,u) m(1,v)\,,
$$
which is d'Alembert's functional equation implying that 
$$
m(1,u)=\cosh \alpha u
$$
with some complex $\alpha$. Finally, let $y=\frac{1}{x}$ and $v=0$ in \eqref{expL}, then we have
$$
m(1,u)=m(x,u) \,m(\frac{1}{x},0)\,.
$$
Combining our results we get
$$
m(x,u)=|x|^{\lambda} \,\cosh \alpha u\,,
$$
and substitution into \eqref{expL} gives $\alpha=0$. We have proved the following theorem.

\begin{theorem}\label{sinecoset}
Every exponential function on the double coset hypergroup $G/ /K$ has the form
\begin{equation}\label{doucos}
m(x,u)=|x|^{\lambda}
\end{equation}
where $\lambda$ is an arbitrary complex number, and conversely, every function of this form is an exponential on $G/ /K$. Here $m(x,u)$ denotes the value of the exponential $m$ on the double coset of the element $(x,u)$ in $G$.
\end{theorem}

In other words, the hypergroup $L$ has the exponential family $\Phi:L\times \C\to\C$ given by
$$
\Phi(x,u,\lambda)=|x|^{\lambda}\,.
$$ 

Now we describe all sine functions on $L$.

\begin{theorem}\label{sineL}
Let $m_{\lambda}$ be the exponential given by equation \eqref{doucos}. Then every $m_{\lambda}$-sine function $f$ on the double coset hypergroup $G/ /K$ has the form
\begin{equation}\label{dousin}
f(x,u)=c |x|^{\lambda} \ln |x|
\end{equation}
where $c,\lambda$ are arbitrary complex numbers, and conversely, every function of this form is an $m_{\lambda}$-sine function on $G/ /K$. Here $f(x,u)$ denotes the value of the function $f$ on the double coset of the element $(x,u)$ in $G$.
\end{theorem}

\begin{proof}
It is easy to check that the function $f$ given in \eqref{dousin} is an $m_{\lambda}$-sine function for each complex number $c$.
\vskip.3cm

Now let $f$ be an $m_{\lambda}$-sine function on $G/ /K$, that is $f:G\to \C$ is a continuous function satisfying the compatibility conditions \eqref{comp} and 
\begin{equation}\label{dousin1}
f(x y,x v+u)+f(x y,x v-u)=2 f(x,u) |y|^{\lambda}+ 2 f(y,v) |x|^{\lambda}
\end{equation}
for each $(x,u), (y,v)$ in $G$. Substituting $y=1, u=0$ we have $f(1,0)=0$, and substituting $x=1$ we obain the functional equation
\begin{equation}\label{dousin2}
f(1,u+v)+f(1,u-v)=2 f(1,u) + 2 f(y,v)
\end{equation}
for each $u,v$ in $\R$ and $y\ne 0$. In particular, by putting $y=1$, it follows that the function $u\mapsto f(1,u)$ satisfies the square norm functional equation, hence, by continuity, we infer that $f(1,u)= a u^2$ with some complex number $a$. Substituting back into \eqref{dousin2} we have
$$
f(1,u+v)+f(1,u-v)=2 a u^2 + 2 f(y,v)\,,
$$
and here the left side is symmetric in $u$ and $v$. Consequently, we conclude that
$$
f(y,u)=f(y,0)+a u^2 |y|^{\lambda}
$$
for each $u,y$ in $\R$ with $y\ne 0$. Substitution into the original equation gives $a=0$, hence $f$ does not depend on the second variable. In this case equation \eqref{dousin1} can be divided by $|x y|^{\lambda}$ and we have that $f$ satisfies
$$
\frac{f(x y,0)}{|x y|^{\lambda}}=\frac{f(x, 0)}{|x |^{\lambda}}+\frac{f(y,0)}{|y|^{\lambda}}\,,
$$ 
hence, by continuity, $f(x,u)=f(x,0)= c |x|^{\lambda} \ln |x|$ and the thoerem is proved.
\end{proof}

Finally, we remark that our present results strongly support the conjecture that in the presence of an exponential family on a hypergroup all sine functions can be described in a similar way as in theorems \ref{polyone}, \ref{Su2}, \ref{sinsev}, \ref{Sturmsine2} and \ref{sineL}.


\begin{thebibliography}{1}




\bibitem{BlH95}
W.~R. Bloom and H.~Heyer, {\rm Harmonic analysis of probability measures on hypergroups}, de Gruyter Studies in Mathematics, vol.~20, Walter de Gruyter {\&} Co., Berlin, 1995.

\bibitem{HeR79}
E.~Hewitt and K.A.~Ross, {\rm Abstract harmonic analysis. {V}ol. {I}}, Fundamental
Principles of Mathematical Sciences vol.~115, Springer-Verlag, Berlin, 1979. 


\bibitem{MR2042564}
L.~Sz{\'e}kelyhidi,
\emph{Functional equations on hypergroups},
\newblock In {\em Functional equations, inequalities and applications}, 
167--181. Kluwer Acad. Publ., Dordrecht, 2003.

\bibitem{MR2107959}
{\'A}.~Orosz and L.~Sz{\'e}kelyhidi,
\emph{Moment functions on polynomial hypergroups in several variables}, 
Publ. Math. Debrecen, \textbf{65(3-4)}(2004), 429--438.

\bibitem{MR2161803}
{\'A}.~Orosz and L.~Sz{\'e}kelyhidi,
\emph{Moment functions on polynomial hypergroups},
Arch. Math., \textbf{85(2)}(2005), 141--150. 









\bibitem{Sze12}
L.~{S}z{\'e}kelyhidi, \emph{Functional Equations on Hypergroups}, World Scientific Publishing Co. Pte. Ltd.,  New Jersey, London, 2012.


\bibitem{FeSz15}
Z.~Fechner and L.~Sz\'ekelyhidi, \emph{Sine and cosine equations on hypergroups}, (to appear)

\end{thebibliography}
\end{document}